\documentclass[12pt]{article}

\usepackage{amsmath,amsthm,amssymb,amsfonts,amscd}

\usepackage{xcolor}

\newtheorem{theorem}{Theorem}[section]
\newtheorem{corollary}[theorem]{Corollary}

\newtheorem{definition}[theorem]{Definition}

\def\Bbb#1{{\fam\msbfam\relax#1}}
\font\fivemsb=msbm5
\font\sevenmsb=msbm7
\font\tenmsb=msbm10
\newfam\msbfam
\textfont\msbfam=\tenmsb
\scriptfont\msbfam\sevenmsb
\scriptscriptfont\msbfam\fivemsb

\def\Om{\Omega}
\def\om{\omega}

\def\s{\sigma}

\def\l{\lambda}
\def\m{\mu}
\def\r{\rho}

\def\la{\langle}
\def\ra{\rangle}

\def\wt{\widetilde}

\def\mmm{\!\!\!}
\def\sqbox{{\vcenter{\hrule height .5pt \hbox{\vrule
   width .5pt height 5pt \kern 5pt \vrule width .5pt}
   \hrule height .5pt}}}

\def\cab{\Omega}
\def\galgal#1{\langle{#1}\rangle}

\def\mmm{\!\!\!}

\begin{document}


\title{Wiener integrals with respect to  Yeh processes}

\author{Jae Gil Choi\thanks{{\color{red}This work was supported by
    the Korea Research Foundation Grant. KRF-2005-214-C00004.}}\\
        Department of Mathematics\\
        Louisiana State University\\
        Baton Rouge, LA 70803, USA\\
        jgchoi@math.lsu.edu}

\date{{\color{red}July 2006}}
\maketitle

\begin{abstract}
We define  Wiener integrals
with respect to Yeh processes and study their
properties. In particular, we obtain the
martingale property of the associated stochastic
processes and give a series expansion of Wiener
integrals with respect to centered Yeh process.
Moreover, we derive a representation of an Yeh
process in terms of a random series.

\par\vskip .1 true in 


\par\vskip .1 true in 


 \par\vskip .1 true in 

\end{abstract}


\setcounter{equation}{0}
\section{Introduction} \label{sec:1}

\par
The theory of stochastic integrals and stochastic differential
equations was initiated and developed by K. It\^o \cite{Ito44}
\cite{Ito51O}. There has been a tremendous amount of papers
and books in the literature on the It\^o theory.
For an elementary introduction, see the recent book \cite{Kuo06}.


\par
Let $B(t), \,t\geq 0,\,\om\in\Om,$ be a Brownian motion
and let $[a, b]\subset [0, \infty)$ be a finite interval.
Since with probability one the function $t \mapsto B(t)$ is
nowhere       differentiable,
the integral $\int_a^b f(t)\,dB(t)$ can be defined pathwise
by the ordinary calculus only for a very small class of
deterministic functions $f(t)$. However, by using the special
properties of a Brownian motion, we can define the
Wiener integral $\int_a^b f(t)\,dB(t)$ for any deterministic
function $f$ in $L^2[a, b]$. Moreover, the Wiener integral
can be extended to the It\^o stochastic integral
$\int_a^b f(t)\,dB(t)$
for stochastic processes $f(t, \om)$ satisfying certain
conditions (see Chapters 4 and 5 in \cite{Kuo06}).


\par
In this paper we will extend the Wiener integral
from a Brownian motion to a more general stochastic process
defined in \cite{Yeh}, which we call an Yeh process. An
{\em Yeh process} on $[a, b]$ is a continuous additive
stochastic process $X(t, \om),\, t\in [a, b], \,\om\in\Om$,
such that for any $a\leq s< t\leq b$,
$$
 X(t)-X(s) \sim N\big(\l(t)-\l(s), \,\r(t)-\r(s)\big),
$$
where $N(\m, \s^2)$ denotes the normal distribution
with mean $\m$ and variance $\s^2$,
$\lambda$ is a continuous real-valued function
on  $[a,b]$  
and $\rho(t)$ is a continuous monotonically increasing
real-valued function on $[a, b]$. 
Thus an Yeh process
is determined by the functions $\l(t)$ and $\r(t)$.
We will further assume throughout this paper that
$\lambda(t)$ is a function of bounded variation
on $[a,b]$ and the measure defined by $\r(t)$ is
equivalent to the Lebesgue measure on $[a, b]$.
These conditions are weaker than those in the paper \cite{CCS}. 
In particular, the function $\lambda(t)$ in \cite{CCS}
is assumed to be absolutely continuous with $\lambda' \in L^2[a,b]$.
Thus we can take $\lambda(t)$ to be the Cantor function
in this paper, but not in \cite{CCS}.


\par
Note that when $\l(t)=0$ and $\r(t)=t$, the Yeh process
is a Brownian motion. On the other hand, we need to
point out that a Brownian motion  is stationary in time,
while in general an Yeh process is not stationary in time
and is subject to a shift $\lambda(t)$.


\par
Suppose $X(t)$ is an Yeh process associated with functions
$\l(t)$ and $\r(t)$ on $[a, b]$. Then we have the following
equalities:
\begin{eqnarray}
 E[X(t,\cdot)] &\mmm=\mmm&\lambda(t),
    \quad a\le t\le b,   \label{eq:1-1}\\
\noalign{\vskip .04 true in}
 E[X(s, \cdot)X(t,\cdot)] &\mmm=\mmm& \rho (s) +
   \lambda(s)\lambda(t), \quad a\le s  \le  t\le b. \label{eq:1-2}
\end{eqnarray}


\par
Next we define two Hilbert spaces needed in this paper.
Let $L_{\rho}^2[a,b]$ be the Hilbert space of functions
on $[a, b]$ given by
$$
L_{\rho}^2[a,b] =\bigg\{ f \, : \,
\int_{a}^{b} |f  (t)|^2 \,d  \rho (t) <\infty \bigg\}
$$
equipped with the inner product defined by
$$
\galgal{f, g}_{\rho}=\int_a^b f(t)g(t) \,d  \rho (t).
$$
Note that by the assumption on $\r(t)$, we have
$L_{\rho}^2[a,b] = L^2[a,b]$ as sets and the norm
$\|\cdot\|_{\r}$ is equivalent to the $L^2[a,b]$-norm
$\|\cdot\|_2$. 
Similarly, let
$$
L_{\lambda, \rho}^{2}[a,b]
  =\bigg\{  f\in L_{\rho}^{2}[a,b]:
  \int_a^b |f(t)|^2 \,d | \lambda |(t)< \infty\bigg\},
$$
where $| \lambda| $ is the total variation function of $\lambda$.
Then $L_{\lambda,\rho}^2[a,b]$  is a Hilbert space with
the inner product defined by
$$
\galgal{f,g}_{\lambda,\rho}
=\int_a^b f(t)g(t)\, d [ \rho (t) + | \lambda |(t)].
$$
It is easy to see that $\|f\|_{\lambda,\rho}=0$
if and only if $f=0$ for $m_{\rho}$-a.e. and   $f=0$ for $ m_{|\lambda|}$-a.e. 
where $m_{\rho}$ and $m_{|\lambda|}$ are Lebesgue-Stieltjes measures 
induced by $\rho$ and $|\lambda|$, respectively.


\setcounter{equation}{0}
\section{Wiener integral with respect to an Yeh process}

Let ${\cal S}[a,b]$ be the set of all step functions
on $[a,b]$,
\begin{equation}  \label{eq:2-1}
f=\sum_{i=1}^n c_{i} 1_{[t_{i-1}, t_i)} ,
\end{equation}
where $a=t_0 < t_1< \cdots < t_n=b$ and $c_{i} \in
\Bbb R$. Obviously, ${\cal S}[a,b]$ is a dense subspace
of $L_{\lambda, \rho}^{2}[a,b]$.


\par
For a step function $f(t)$ represented by Equation
(\ref{eq:2-1}), we define the Wiener integral $I(f)$
with respect to an Yeh process $X(t)$ to be the
random variable
\begin{equation*}
I(f)(\om) =\sum\limits_{i=1}^n c_i \big(X(t_i, \om)
   -X(t_{i-1}, \om)\big), \quad  \om \in \Om.
\end{equation*}
It is easy to check that $I(f)$ is well-defined,
namely, $I(f)$ is independent of the representation
of $f$ in Equation (\ref{eq:2-1}). Moreover,
$I(\alpha f +\beta g)=\alpha I(f)+\beta I(g)$ for any
$\alpha, \beta\in \Bbb R$ and $f, g \in {\cal S}[a,b]$.


\par
Using Equations (\ref{eq:1-1}) and (\ref{eq:1-2}), and the same 
ideas as in the proof of Lemma 2.3.1 in \cite{Kuo06},
we have the following theorem.


\begin{theorem} \label{thm:2-1}
  For $f,g \in {\cal S}[a,b]$, the following hold:
\begin{itemize}
\item[{\rm (1)}] $E[I(f)]=\int_a^b f(t) \,d \lambda(t)$,
\item[{\rm (2)}] $E[I(f) I(g)] = \int_a^b f(t)g(t)\,d \r(t)
    +  \int_a^b f(t)\,d \l(t) \int_a^b g(t)\,d \l(t)$,
\item[{\rm (3)}] $E[(I(f))^2]=\int_a^b f(t)^2\,d \r(t)
    + \big(\int_a^b f(t)\,d \l(t)\big)^2$,
\item[{\rm (4)}] $I(f)$ has normal distribution
$N\big(\int_a^b f(t) \,d \lambda(t), \,
\int_a^b f(t)^2 \,d \rho(t)\big)$.
\end{itemize}
\end{theorem}
 

\par
Next we extend the Wiener integral $I(f)$ from ${\cal S}[a,b]$
to $L_{\lambda, \rho}^{2}[a,b]$. Let $f \in L_{\lambda,
\rho}^{2}[a,b]$. By the denseness of ${\cal S}[a,b]$ in
$L_{\lambda, \rho}^{2}[a,b]$, there exists a sequence
$\{f_n\}_{n \in \Bbb N}$
in $L_{\lambda, \rho}^{2}[a,b]$
such that $\lim_{n\to \infty}\|f-f_n\|_{\lambda, \rho}=0$.
Then by the linearity of the mapping $I$ and assertion (3)
of Theorem \ref{thm:2-1}, we have
$$
\aligned
&\|I(f_m)-I(f_n)\|_2^2\\
&\le\int_a^b |f_n(t)-f_m(t)|^2\,d \r(t)
    + \bigg(\int_a^b \big|f_n(t)-f_m(t)\big|\,d |\l|(t)\bigg)^2 \\
&\le\int_a^b |f_n(t)-f_m(t)|^2\,d \r(t)
    + \bigg(\int_a^b \, d |\l|(t)\bigg)
\bigg(\int_a^b \big|f_n(t)-f_m(t)\big|^2\,d |\l|(t)\bigg) \\
&\le \big(1+ |\l|(b)-|\l|(a)\big)\|f_n-f_m\|_{\l,\r}^2.
\endaligned
$$
Hence $\{I(f_n)\}$ is a Cauchy sequence in $L^2(\cab)$ and
so it converges in $L^2(\cab)$.
Define
\begin{equation}    \label{eq:2-3}
I(f)=\lim\limits_{n\to\infty}I(f_n), \quad
\hbox{ in } L^2(\cab).
\end{equation}


\par
It is easy to check that $I(f)$ is independent of the
choice of the sequence $\{f_n\}_{n \in \Bbb N}$. Thus we
can make the following definition.


\begin{definition}
Let $f \in L_{\lambda, \rho}^{2}[a,b]$.
The limit $I(f)$ defined by Equation {\rm (\ref{eq:2-3})}
is called  the {\em Wiener  integral  of $f$ with respect
to the Yeh process  $X(t)$}. The Wiener integral $I(f)$
 will be denoted by
$$
I(f)(\om) =\bigg( \int_a^b f(t) \,dX(t)\bigg)(\om),
   \quad \hbox{\rm for a.s. } \om\in \cab.
$$
\end{definition}


\begin{theorem} \label{thm:a}
The Wiener integral $I(\cdot)$ is a linear mapping
from $L_{\lambda, \rho}^{2}[a,b]$ into $L^2(\Om)$.
Moreover, the assertions {\rm (1), (2), (3)}, and
{\rm (4)} in Theorem {\rm \ref{thm:2-1}} hold for any
$f, g \in L_{\lambda, \rho}^{2}[a,b]$.
\end{theorem}


\par
In particular, for any $f, g \in L_{\lambda, \rho}^{2}
[a,b]$, we have the following equality which will be
used later.
\begin{equation} \label{eq:2-5}
 E[I(f)I(g)] = \int_a^b f(t) g(t) \,d \rho(t)
  + \int_a^b f(t) \,d \lambda(t) \int_a^b g(t)
  \,d \lambda(t).
\end{equation}


\begin{corollary}
Let $f, g\in L_{\lambda, \rho}^{2}[a,b]$.
Then $\galgal{f,g}_{\rho} = 0$ if and only if
the Gaussian random variables $I(f)$ and $I(g)$
are independent.
\end{corollary}


\par
The next theorem relates the Wiener integral of
a function $f$ of bounded variation to the pathwise
Riemann-Stieltjes integral of $f$. 
Using the same ideas as in the proof of Theorem 2.3.7 in \cite{Kuo06},
we have the following theorem.


\begin{theorem}
Let $f$ be a function of bounded variation on $[a,b]$.
Then
$$
I(f)(\om)=(RS)\int_a^b f(t)\,dX(t, \om), \quad
 {\rm ~a.s.~} \om \in \cab,
$$
where the right hand side is a Riemann-Stieltjes integral
for each sample path of $X(t)$.
\end{theorem}


\setcounter{equation}{0}
\section{Properties of Wiener integrals}

It is well known that a Brownian motion $B(t)$ is a
martingale with respect to the filtration
$\{{\cal F}_t^{B} : t\geq 0\}$ given by
${\cal F}_t^{B} =\sigma\{B(s) : a\le s\le t\}$.
Moreover, for any $f\in L^2[a,b]$, the stochastic
process
$$
M(t)=\int_a^t f(s) \,dB(s), \quad t\in [a,b],
$$
is also a martingale with respect to $\{ {\cal F}_t^{B}\}$.
However, an Yeh process $X(t)$ determined by $\lambda$
and $\rho$ may not be a martingale with respect to the
filtration $ {\cal F}_t =\sigma\{X(s) : a\le s\le t\},\,
a\leq t\leq b$. In fact,
for any $ a \leq s \leq t\leq b$, we have
$$
 E[X(t)|{\cal F}_s]  = E [ X ( t)-X(s) ]+ X ( s)
 =\lambda (t)-\lambda (s)+ X(s).
$$
Hence if $\lambda$ is an increasing function on $[a,b]$,
then $X(t)$  is
a submartingale with respect to $\{{\cal F}_t\}$. But if
$\lambda$ is a decreasing function on $[a,b]$, then
$X(t)$ is a supermartingale with respect to $\{{\cal F}_t\}$.


\begin{theorem} \label{thm:3-1}
Suppose the mean function $\lambda$ of an Yeh process
$X(t),\, a\leq t\leq b$, is increasing on $[a,b]$ and let
$f\in  L_{\lambda, \rho}^{2}[a,b] $ be a nonnegative
function. Then the stochastic process
\begin{equation}\label{eq:3-1}
M(t)=\int_a^t f(s) \,dX(s), \quad a\le t \le b,
\end{equation}
is a submartingale with respect to the filtration
$\{{\cal F}_t : a\leq t\leq b\}$ defined by
${\cal F}_t=\sigma\{X(s): a\le s\le t\}, \, a\leq t\leq b$.
\end{theorem}
\begin{proof}
First we show that $E|M(t)|<\infty$ for all $t\in [a,b]$
in order to take conditional expectation of $M(t)$.
Apply Equation (\ref{eq:2-5}) with $f=g$ to get
\begin{eqnarray*}
E\big[\big|M(t)\big|^2\big]
& \mmm=\mmm&\int_a^t  f(s)^2 \,d \rho(s) +\bigg(\int_a^t f(s)
\,d \lambda(s)\bigg)^2\\
& \mmm\le \mmm & \int_a^b  f(s)^2 \,d \rho(s) +\bigg(\int_a^b f(s)
\,d \lambda (s)\bigg)^2.\\
\end{eqnarray*}
Hence $E|M(t)|\le \big\{E[|M(t)|^2]\big\}^{1/2}<\infty$.
Next we need to show that
\begin{equation} \label{eq:3-a}
 E[M(t)| {\cal F}_s]\ge M(s), \quad {\rm almost~ surely},
\end{equation}
for any $a\leq s\le t\leq b$. Note that for any $s<t$,
$$
M(t)=M(s) + \int_s^t f(u) \,dX(u).
$$
Hence we have
$$
E[M(t)| {\cal F}_s]=M(s)
+E \bigg[\int_s^t f(u) \,dX(u)\bigg| {\cal F}_s\bigg].
$$
Thus in order to prove Equation (\ref{eq:3-a}),
it suffices to show that for any $s\le t$,
\begin{equation} \label{eq:3-2}
E \bigg[\int_s^t f(u) \,dX(u)\bigg| {\cal F}_s\bigg] \ge 0.
\end{equation}


\par
First suppose $f$ is a nonnegative  step function
represented by
$$
 f=\sum_{i=1}^n c_i 1_{[t_{i-1}, t_i)},
$$
where $t_0=s$ and $t_n=t$. In this case, we have
$$
\int_s^t f(u) \,dX(u)
=\sum\limits_{i=1}^n c_i \big(X(t_{i})-X (t_{i-1}) \big),
\quad c_i\ge0.
$$
But $X(t_i)-X(t_{i-1})$, $i=1,\ldots,n,$ are all independent
of the $\sigma$-field ${\cal F}_s$. Hence
\begin{eqnarray*}
E \bigg[\int_s^t f(u) \,dX(u)\bigg| {\cal F}_s\bigg]
&\mmm=\mmm& \sum\limits_{i=1}^n c_i E\big[X(t_i)-X(t_{i-1})
\big|{\cal F}_s\big]\\
&\mmm=\mmm&\sum\limits_{i=1}^n c_i E\big[X(t_i)-X(t_{i-1})\big] \\
&\mmm=\mmm &\sum\limits_{i=1}^n c_i\big(\lambda(t_i)-\lambda(t_{i-1})\big).
\end{eqnarray*}
Thus Equation (\ref{eq:3-2}) holds for any
nonnegative  step function $f$.


\par
Next suppose $f\in L_{\lambda, \rho}^{2}[a, b]$ and $f\ge 0$.
Choose a sequence $\{f_n\}_{n=1}^{\infty}$ of nonnegative
step functions
converging to $f$ in $ L_{\lambda, \rho}^{2}[a, b]$
monotonically. Then by the conditional Jensen's inequality,
we have the inequality
$$
 \big| E[X|{\cal F}]\big|^2\le E[X^2 |{\cal F}],
$$
which implies that
\begin{eqnarray*}
\lefteqn{ \bigg|  E \bigg[\int_s^t \big(f_n (u)
     -f(u)\big)\,dX(u)\bigg| {\cal F}_s\bigg] \bigg|^2}\\
&\le \mmm & E \bigg[\bigg(\int_s^t \big(f_n (u)-f(u)\big)
\,dX(u)\bigg)^2\bigg|
{\cal F}_s\bigg] .
\end{eqnarray*}
Moreover, we use the property $E\big[E[X|{\cal F}]\big]=E[X]$
of conditional expectation and then apply Equation
(\ref{eq:2-5}) with $f=g$ to get
\begin{eqnarray*}
\lefteqn{ E\bigg[\bigg|
E \bigg[\int_s^t \big(f_n (u)-f(u)\big)\,dX(u) \bigg|
   {\cal F}_s\bigg] \bigg|^2\bigg]}\\
&\le\mmm & E\bigg[ E \bigg[\bigg(\int_s^t
\big(f_n (u)-f(u)\big)\,dX(u)\bigg)^2\bigg|
   {\cal F}_s\bigg] \bigg]\\
& =\mmm &\int_s^t \big(f_n (u)-f(u)\big)^2 \,d \rho(u)
   +\bigg(\int_s^t \big(f_n (u)-f(u)\big)\,d \lambda(u)\bigg)^2\\
&\le \mmm & \int_a^b \big(f_n (u)-f(u)\big)^2 d\rho(u)
   +\bigg(\int_a^b \big|f_n (u)-f(u)\big|\,d |\lambda|(u)\bigg)^2\\
&\le \mmm & \big(1+|\l|(b)-|\l|(a)\big)\|f_n-f_m\|_{\l,\r}^2\\
\noalign{\vskip .04 true in}
& \to \mmm & 0,
\end{eqnarray*}
as $n\to \infty$.
This shows that the sequence $E [\int_s^tf_n (u)\,dX(u)|
{\cal F}_s ], \,n\geq 1,$ of random variables converges to
$E [\int_s^t f(u)\,dX(u)| {\cal F}_s ]$ in $L^2 (\cab)$.
Note that the convergence of a sequence in $L^2 (\cab)$
implies convergence in probability, which implies the
existence of a subsequence converging  almost surely.
Thus by choosing a subsequence, if necessary, we
conclude that the following equality holds with
probability one,
\begin{equation}\label{eq:3-3}
\lim\limits_{n\to \infty}
  E \bigg[\int_s^tf_n (u)\,dX(u)\bigg| {\cal F}_s \bigg]
= E \bigg[\int_s^tf  (u)\,dX(u)\bigg| {\cal F}_s \bigg].
\end{equation}
But $E[\int_s^t f_n(u) \,dX(u) | {\cal F}_s ] \ge 0$
since we have already shown that Equation (\ref{eq:3-2})
holds for nonnegative step functions.
Hence by Equation (\ref{eq:3-3}),
$$
E\bigg[\int_s^tf  (u)\,dX(u)\bigg| {\cal F}_s \bigg]\ge 0,
$$
which shows that the inequality in Equation (\ref{eq:3-2})
holds for any
nonnegative function $f$ in
$ L_{\lambda, \rho}^{2}[a,b]$.
\end{proof}


\par
From the proof of the above theorem, we get
the following assertion under various conditions on the
mean function $\l(t)$ and the integrand $f(t)$:
\begin{itemize}
\item[(1)] If the mean function $\lambda(t)$ of an Yeh process $X(t)$
is increasing  on $[a,b]$ and $f \in L_{\lambda, \rho}^{2}[a,b]$
is nonpositive,
then  the stochastic process $M(t)$ given by Equation (\ref{eq:3-1})
is a supermartingale.
\item[(2)] If the mean function $\lambda(t)$ of an Yeh process $X(t)$
is decreasing on $[a,b]$ and $f \in L_{\lambda, \rho}^{2}[a,b]$ is
nonnegative,
then the stochastic process $M(t)$ given by Equation (\ref{eq:3-1})
is a supermartingale.
\item[(3)] If the mean function $\lambda(t)$ of an Yeh process $X(t)$
 is decreasing on $[a,b]$ and $f \in L_{\lambda, \rho}^{2}[a,b]$
 is nonpositive,
then the stochastic process $M(t)$ given by Equation (\ref{eq:3-1})
is a submartingale.
\end{itemize}


\par
In Theorem \ref{thm:3-1} and the above assertions (1),
(2), and (3), the condition on the positivity or
negativity of the integrand $f$ is necessary.
For example, consider the case $\lambda(t)=t$ on
$[0,1]$. Let $f$ be the following step function
$$
 f(t) = \left\{
\begin{array}{ll}
  1/2, & \mbox{~if $0\leq t < 1/3$;}  \\
\noalign{\vskip .05 true in}
   -1/2, & \mbox{~if $1/3 \leq t < 2/3$;} \\
\noalign{\vskip .05 true in}
   2, & \mbox{~if $2/3 \leq t \leq 1$.}
\end{array}
\right.
$$
Then we have
\begin{eqnarray*}
 E[M(1/2)| {\cal F}_{1/4}] &\mmm=\mmm &
 M(1/4)- 1/{24} \,< \,M(1/4), \\
\noalign{\vskip .05 true in}
 E[M(3/4)|{\cal F}_{1/4}] & \mmm=\mmm &
 M(1/4)+1/24 \,> \, M(1/4).
\end{eqnarray*}
Thus the stochastic process $M(t)$ in Equation
(\ref{eq:3-1}) given by the above function $f(t)$
is neither a submartingale nor a supermartingale.


\setcounter{equation}{0}
\section{Random series expansion of Wiener integrals}

Let $X(t)$ be an Yeh process with mean function
$\l(t)$ and variance function $\r(t)$. The {\em
centered Yeh process} $\wt X(t)$ is defined by
$$
 \wt X(t) = X(t) - \l(t), \quad a\leq t\leq b.
$$
Thus $\wt X(t)$ is an Yeh process with mean function
$0$ and variance function $\r(t)$. We will use
$\wt I(f)$ to denote the Wiener integral of
$f\in L_{\r}^2[a, b]$ with respect to $\wt X(t)$.
Obviously, we have the equality
$$
 \wt I(f) = I (f) - \int_a^b f(t)\,d\l(t), \quad
  f\in L_{\l, \r}^{ 2}[a, b].
$$
Moreover, by Theorem \ref{thm:a}, $\wt I(f)$ is a
Gaussian random variable and
$$
 E[\wt I(f)] = 0, \quad E[\wt I(f) \wt I(g)]
  = \la f, g\ra_{\r}.
$$
Therefore, $\wt I(f)$ and $\wt I(g)$ are independent
if and only if $\la f, g\ra_{\r}=0$.


\par
Let $\{\phi_n\}_{n=1}^{\infty}$ be an orthonormal basis
for the Hilbert space $L_{\rho}^{2}[a,b]$.
Each $f\in L_{\rho}^{2}[a,b]$ has the
 following expansion
\begin{equation}\label{eq:5-4}
f =\sum\limits_{n=1}^{\infty}\galgal{f,\phi_n}_{\rho}\,
 \phi_n.
\end{equation}
Moreover, we have the Parseval identity
$\|f\|_{\rho}^2=\sum_{n=1}^{\infty}
 \galgal{f,\phi_n}_{\rho}^2.
$


\par
If we informally take the Wiener integral with
respect to $\wt X(t)$ in both
sides of Equation (\ref{eq:5-4}), then we would get
\begin{equation}\label{eq:5-5}
\int_a^b f(t) \,d\wt X(t) =\sum\limits_{n=1}^{\infty}
 \galgal{f,\phi_n}_{\rho}
\int_a^b \phi_n (t) \,d\wt X(t).
\end{equation}
We claim that this equality is indeed true in the
$L^2(\cab)$ sense. To prove this claim, use
Equation (\ref{eq:2-5}) to show that
\begin{eqnarray*}
\lefteqn{
E\bigg[\bigg(\wt I(f)
 -\sum\limits_{n=1}^{N}\galgal{f,\phi_n}_{\rho}
    \, \wt I(\phi_n) \bigg)^2\bigg]} \\
&=\mmm & E\bigg[\bigg(
  \wt I\Big(f - \sum\limits_{n=1}^{N}
   \galgal{f,\phi_n}_{\rho}\, \phi_n\Big)
  \bigg)^2\bigg]  \\
&=\mmm & \bigg\|f - \sum\limits_{n=1}^{N}
     \galgal{f,\phi_n}_{\rho}\,\phi_n\bigg\|_{\rho}^2 \\
& \to \mmm & 0,
\end{eqnarray*}
as $N \to 0$. Hence the random series in Equation
(\ref{eq:5-5}) converges in $L^2(\cab)$
to the random variable in the left-hand side of
Equation (\ref{eq:5-5}).
But the $L^2(\cab)$ convergence implies convergence
in probability. On the other hand, note that
the random variables
$\wt I(\phi_n),\, n\geq 1,$ are independent. Hence
we can apply
the L\'evy equivalence theorem (page 173 \cite{ito})
to conclude that
the random series in Equation (\ref{eq:5-5}) converges
almost surely. Thus we have proved the next theorem
for the random series expansion of Wiener integral
with respect to the centered Yeh process
$\wt X(t) = X(t) - \l(t)$.


\begin{theorem}\label{thm:3-3}
Let $\{\phi_n\}_{n=1}^{\infty}$ be an orthonormal basis
for $L_{\rho}^{2}[a,b]$.
Then for each $f\in L_{  \rho}^{2}[a,b]$,
the Wiener integral of $f$ with respect to $\wt X(t)$
has the following random series expansion,
\begin{equation} \label{eq:4-2}
\int_a^b f(t) \,d\wt X(t)=\sum\limits_{n=1}^{\infty}
 \galgal{f,\phi_n}_{\rho}
\int_a^b \phi_n (t) \,d\wt X(t), 
\end{equation}
where the right hand side converges in $L^2(\Om)$ and
almost surely.
\end{theorem}


\par
It follows from Equation (\ref{eq:4-2}) that we also
have the equality for Wiener integral with respect
to the Yeh process $X(t)$,
$$
\int_a^b f(t) \,dX(t)= \int_a^b f(t)\,d\l(t) +
\sum\limits_{n=1}^{\infty}
 \galgal{f,\phi_n}_{\rho}
\int_a^b \phi_n (t) \,d\wt X(t).
$$
In particular, take the function $f=1_{[a,t)}$. Then
we have the random series representations of $\wt X(t)$
and  $X(t)$ by:
\begin{eqnarray*}  
 \wt X(t)
  &\mmm=\mmm&\sum\limits_{n=1}^{\infty}
   \bigg(\int_a^t \phi_n (s) \,d \rho(s)\bigg)
    \bigg(\int_a^b \phi_n (s)\,d\wt X(s)\bigg),
  \nonumber\\
X(t) &\mmm=\mmm& \l(t) + \sum\limits_{n=1}^{\infty}
   \bigg(\int_a^t \phi_n (s) \,d \rho(s)\bigg)
    \bigg(\int_a^b \phi_n (s)\,d\wt X(s)\bigg).
\end{eqnarray*}


\par
Note that the sequence $\wt I(\phi_n) =
\int_a^b \phi_n(s)\,d\wt X(s),\, n\geq 1$, is
an independent sequence of standard normal
random variables. Thus, given a function $\rho(t)$
satisfying the conditions in Section \ref{sec:1},
we can consider the random series
$$
\wt X(t) = \sum\limits_{n=1}^{\infty}
   \bigg(\int_a^t \phi_n (s) \,d \rho(s)\bigg)
   \,\xi_n,
$$
where $\{\phi_n : n\geq 1\}$ is an orthonormal
basis for $L_{\rho}^2[a, b]$ and
$\{\xi_n : n\geq 1\}$
is an independent sequence of standard Gaussian random
variables. It can be checked that this random
series indeed converges in $L^2(\Om)$ and almost
surely and that the stochastic process $\wt X(t)$ is
an Yeh process with mean function $0$ and variance
function $\rho(t)$. In addition, if we are also
given a function $\l(t)$ satisfying the conditions
in Section \ref{sec:1}, then the following random
series
$$
 X(t) = \l(t) + \sum\limits_{n=1}^{\infty}
   \bigg(\int_a^t \phi_n (s) \,d \rho(s)\bigg)
   \,\xi_n,
$$
is an Yeh process with mean function $\l(t)$ and
variance function $\rho(t)$.

\bigskip

\centerline{\bf Acknowledgments}
The  author  wishes to express his gratitude
to Professor Hui-Hsiung Kuo for his encouragement
and valuable advice as well as to the Louisiana State
University for its hospitality.



\begin{thebibliography}{17}


\bibitem{CCS}
     S.J. Chang,   J.G. Choi  and  D. Skoug,
     {\it Integration by parts formulas involving generalized
     Fourier-Feynman transforms on function space},
     Trans. Amer. Math. Soc. {\bf 355} (2003), 2925--2948.
\bibitem{Ito44}
     K. It\^o,
     {\it  Stochastic integral},
     Proc. Imp. Acad. Tokyo {\bf 20} (1944), 519-524.
\bibitem{Ito51O}
     K. It\^o,
     {\it  On Stochastic Differential Equations},
     Memoir, Amer. Math. Soc., vol {\bf 4}, 1951.

\bibitem{ito} 
     K. It\^o, 
     {\it Introduction to Probability Theory},
     Cambridge University Press, 1978.

\bibitem{Kuo06}
     H.-H. Kuo,  
     {\it  Introduction to Stochastic  Integration},
     Universitext (UTX), Springer,  2006.
\bibitem{Yeh}
     J. Yeh, 
     {\it  Stochastic Processes  and the Wiener Integral},
     Marcel Dekker, Inc., New York, 1973.

\end{thebibliography}
\end{document}